\documentclass[11pt]{amsart}
\usepackage{amsmath, amsthm, amssymb}
\usepackage[all]{xy}
\usepackage{mathtools}
\usepackage{enumitem}
\usepackage{mathrsfs}
\usepackage{tikz}
\usepackage{skak}
\usepackage{graphics}
\usepackage{array}

\newcommand{\C}{\mathcal{C}}

\newcommand{\mcH}{\mathrel{\mathcal{H}}}
\newcommand{\mcR}{\mathrel{\mathcal{R}}}
\newcommand{\mcL}{\mathrel{\mathcal{L}}}
\newcommand{\mcD}{\mathrel{\mathcal{D}}}

\newcommand{\mcmu}{\mathrel{\mu}}
\newcommand{\ZM}{\text{ZM}}

\DeclareMathOperator{\dom}{dom}

\newtheorem{prop}{Proposition}[section]
\newtheorem{thm}[prop]{Theorem}
\newtheorem{cor}[prop]{Corollary}
\newtheorem{lem}[prop]{Lemma}
\theoremstyle{definition}
\newtheorem{defn}[prop]{Definition}

\newtheorem{exmp}[prop]{Example}
\newtheorem{rem}[prop]{Remark}

\newlist{thmenum}{enumerate}{10}
\setlist[thmenum,1]{label=\textnormal{(\alph*)}}
\setlist[thmenum,2]{label=\textnormal{(\roman*)}}

\begin{document}

\title[Zigzag Maps and Path Categories]{On zigzag maps and the path category of an inverse semigroup}

\author[A. Donsig]{Allan Donsig}
\address{Department of Mathematics\\ 
University of Nebraska-Lincoln\\ 
Lincoln, NE  68588} 
\email{adonsig@unl.edu}

\author[J. Gensler]{Jennifer Gensler}
\address{Department of Mathematics and Statistics\\
California State University, Long Beach\\
1250 Bellflower Blvd\\
Long Beach, CA 90840}

\author[H. King]{Hannah King}
\address{Department of Mathematics\\
Taylor University\\
236 W Reade Ave\\
Upland, IN 46989}

\author[D. Milan]{David Milan}
\address{Department of Mathematics\\
The University of Texas at Tyler\\
3900 University Boulevard\\
Tyler, TX 75799}
\email{dmilan@uttyler.edu}

\author[R. Wdowinksi]{Ronen Wdowinski}
\address{Department of Mathematics\\
Rice University\\
6100 Main St\\
Houston, TX 77005}
\thanks{The second through fifth authors were supported by an NSF grant (DMS-1659221).}

\date{\today}

\begin{abstract} We study the path category of an inverse semigroup admitting unique maximal idempotents and give an abstract characterization of the inverse semigroups arising from zigzag maps on a left cancellative category. As applications we show that every inverse semigroup is Morita equivalent to an inverse semigroup of zigzag maps and hence
the class of Cuntz-Krieger $C^*$-algebras of singly aligned categories include the tight $C^*$-algebras of all countable inverse semigroups, up to Morita equivalence.
\end{abstract}

\maketitle

\section{Introduction}

Inverse semigroups have played a role of increasing prominence in the study of $C^*$-algebras, particularly in the study of graph algebras and their generalizations. If $\Gamma$ is a directed graph, then the graph $C^*$-algebra $C^*(\Gamma)$ is generated by a collection of partial isometries coming from the category of finite paths in $\Gamma$. 

Those partial isometries also generate the graph inverse semigroup $S_\Gamma$ defined by Ash and Hall in \cite{AshHall}. Jones and Lawson characterized graph inverse semigroups as combinatorial proper Perrot semigroups \cite{LawsonGraph}. In particular they showed how to recover the path category of a graph from such a semigroup. To construct the category, it is crucial that every nonzero idempotent lies beneath a unique maximal idempotent.

In this paper, we study zigzag inverse semigroups $\ZM(\C)$ that arise from zigzag maps on a left cancellative category $\C$. We show that $\ZM(\C)$ admits unique maximal idempotents and that one can recover $\C$ as the path category of the semigroup. In Theorem \ref{thm:main} the zigzag inverse semigroups are characterized as the inverse semigroups $S$ with zero satisfying three axioms:
\begin{enumerate}

\item[(Z1)] $S$ admits unique maximal idempotents,

\item[(Z2)] The paths in $S$ generate $S$, and

\item[(Z3)] $S$ is right reductive on domain paths.

\end{enumerate}

The final condition is motivated by work of Cherubini and Petrich on the inverse hull of a right cancellative semigroup \cite{CherubiniPetrich}.  
Briefly, (Z3) requires that, for $s,t \in S$,  if $sx = tx$ for all $x$ in a suitable subset of $S$ (depending on $s$ and $t$), then $s = t$; see Definition~\ref{def:dompath} for details.

Zigzag inverse semigroups appeared in Spielberg's construction of the $C^*$-algebra of a category of paths \cite{spielberg} and more recently in the construction due to B{\'e}dos, Kaliszewski, Quigg, and Spielberg of the $C^*$-algebra of a left cancellative small category \cite{BKQS}. 
Exel and Steinberg \cite{ExelSteinberg} have recently studied a related and even more general class of semigroups, though that construction is not considered in this paper.

As an application of our characterization, we show that every inverse semigroup is Morita equivalent to $\ZM(\C)$ for some left cancellative category $\C$. It follows by a result of Steinberg that the class of Cuntz-Krieger $C^*$-algebras associated with singly aligned left cancellative categories in \cite{BKQS} includes the tight $C^*$-algebras of all countable inverse semigroups up to Morita equivalence.

\section{Preliminaries}
An \emph{inverse semigroup} is a semigroup $S$ such that for each $s$ in $S$ there exists a unique $s^*$ in $S$ such that
\[
	s = ss^*s \quad \text{and}  \quad s^* = s^* s s^*.
\]
The set of idempotents of $S$, denoted $E(S)$, is a commutative subsemigroup of $S$. The natural partial order is defined on $S$ by $s \leq t$ if and only if $s = te$ for some $e \in E(S)$. Green's relations take an especially nice form for inverse semigroups: we have $s \mcL t$ if and only if $s^*s = t^*t$, $s \mcR t$ if and only if $ss^* = tt^*$, and $\mcH \,=\, \mcL \cap \mcR$. Moreover, $s \mcD t$ if and only if there exists $a,b \in S$ such that $a^*a = t^*t$, $aa^* = s^*s$, $b^*b = tt^*$, $bb^* = ss^*$, and $t = b^* s a$. For $e,f \in E(S)$ we have $e \mcD f$ if and only if there exists $a \in S$ with $e = a^*a$ and $f = aa^*$. 

The most important example of an inverse semigroup is the semigroup $I(X)$ of partial bijections on a set $X$. If $g \in I(X)$ with domain $A$ and range $B$ and $f \in I(X)$ with domain $C$ and range $D$, then the product $fg$ is the composition of the functions on the largest possible domain. That is, $fg$ is the bijection of $g^{-1}(B \cap C)$ onto $f(B \cap C)$. The map with empty domain is denoted by $0$. The inverse of $f$ in $I(X)$ is given by $f^{-1}$. For any positive integer $n$, let $I_{n} = I(\{1,2, \dots, n\})$. 

We will use the following conventions for a small category $\C$. The objects of $\C$ are denoted $\C^0$. There are maps $r,s:\C \to \C^0$ called the range and source maps. For $\alpha, \beta \in \C,$ the product $\alpha \beta$ is defined if and only if $s(\alpha) = r(\beta)$. For $\alpha \in \C$, we write $\alpha \C = \{\alpha \beta : \beta \in \C \text{ and }s(\alpha) = r(\beta)\}$. Finally, $\C$ is \emph{left cancellative} if for any $\alpha, \beta, \gamma \in \C$ with $s(\alpha) = r(\beta)$ and $s(\alpha) = r(\gamma)$, $\alpha \beta = \alpha \gamma$ implies that $\beta = \gamma$.  We use LCSC for a left cancellative small category.

There is an equivalence relation on a LCSC  $\C$ defined by $\gamma_1 \sim \gamma_2$ if and only if  $\gamma_1 = \gamma_2 \lambda$ for some invertible $\lambda$. 
It is shown on page 5 of \cite{BKQS} that $\gamma_1 \C = \gamma_2 \C$ if and only if $\gamma_1 \sim \gamma_2$.

One natural example of a left cancellative category is the path category of a directed graph. A directed graph $\Lambda = (\Lambda^0, \Lambda^1, r, s)$ consists of countable sets $\Lambda^0$, $\Lambda^1$ and functions $r,s : \Lambda^1 \to \Lambda^0$ called the \emph{range} and \emph{source} maps, respectively. The elements of $\Lambda^0$ are called \emph{vertices}, and the elements of $\Lambda^1$ are called \emph{edges}. Given an edge $e$, $r_{e}$ denotes the range vertex of $e$ and $s_{e}$ denotes the source vertex. We denote by $\Lambda^*$ the collection of finite directed paths in $\Lambda$. The range and source maps $r,s$ can be extended to $\Lambda^*$ by defining $r_{\alpha} = r_{\alpha_n}$ and $s_{\alpha} = s_{\alpha_1}$ for a path $\alpha = \alpha_n \alpha_{n-1} \cdots \alpha_1$ in $\Lambda^*$. If $\alpha = \alpha_n \alpha_{n-1} \cdots \alpha_1$ and $\beta = \beta_m \beta_{m-1} \cdots \beta_1$ are paths with $s_{\alpha} = r_{\beta}$, we write $\alpha \beta$ for the path $\alpha_n \cdots \alpha_1 \beta_m \cdots \beta_1$. We refer to $\Lambda^*$ as the \emph{path category} of $\Lambda$.

The \textit{graph inverse semigroup} of the directed graph $\Lambda$ is the set
\[ S_{\Lambda} = \{(\alpha,\beta) \in \Lambda^* \times \Lambda^* : s_{\alpha} = s_{\beta} \} \cup \{ 0 \} \]
with products defined by
\[ (\alpha,\beta) (\gamma,\nu) = \left\{\begin{array}{ll}
        (\alpha \gamma', \nu) & \mbox{if $\gamma = \beta \gamma'$} \\
        (\alpha,\nu\beta') & \mbox{if $\beta = \gamma \beta'$} \\
        0 & \mbox{otherwise}
   \end{array} \right. \]
The inverse is given by $(\alpha,\beta)^{*} = (\beta,\alpha)$. 

\section{Zigzag maps and LCSCs} 

In his work generalizing the $C^*$-algebras of higher rank graphs, Spielberg introduced the notion of zigzag maps on a category of paths $\Lambda$ \cite{spielberg}. It was shown in \cite{DonsigMilan} that Spielberg's $C^*$-algebra is isomorphic to the tight $C^*$-algebra of the inverse semigroup $\ZM(\Lambda)$ of zigzag maps on $\Lambda$. These results were recently generalized in \cite{BKQS} and \cite{spielberg2} to the $C^*$-algebras arising from a left cancellative small category (LCSC) $\C$. 

We outline the construction of the inverse semigroup $\ZM(\C)$ of a LCSC $\C$. For more details, see section 7 of \cite{BKQS}. Given $\alpha$ in $\C$, there is a partial bijection $\tau_{\alpha} : s(\alpha) \C \to \alpha \C$ defined by $\tau_{\alpha}(x) = \alpha x$. A \emph{zigzag} in $\C$ is an even tuple
\[
\zeta = (\alpha_1, \beta_1, \alpha_2, \beta_2, \dots, \alpha_n, \beta_n)
\]
where $\alpha_i, \beta_i$ in $\C$ with $r(\alpha_i) = r(\beta_i)$ for $i = 1, \dots, n$ and $s(\beta_i) = s(\alpha_{i+1})$ for $i = 1, \dots, n-1$. Given a zigzag $\zeta = (\alpha_1, \beta_1, \alpha_2, \beta_2, \dots, \alpha_n, \beta_n)$ one defines an associated \emph{zigzag map} $\phi_{\zeta}$ in $I(\C)$ by
\[
	\phi_{\zeta} := \tau_{\alpha_1}^{-1} \tau_{\beta_1} \cdots \tau_{\alpha_n}^{-1} \tau_{\beta_n}. 
\]
The inverse semigroup $\ZM(\C)$ is the subsemigroup of $I(\C)$ consisting of all zigzag maps and the empty function $0$. We refer to $\ZM(\C)$ as a \emph{zigzag inverse semigroup}.

Though this semigroup looks unruly at first, it can take a nice form in certain cases. A LCSC $\C$ is called \emph{singly aligned} if, for every $\alpha$ and $\beta$ in $\C$ such that $\alpha \C \cap \beta \C$ is nonempty, there exists $\gamma$ in $\C$ such that $\alpha \C \cap \beta \C = \gamma \C$. The path category of a directed graph is an example of a singly aligned LCSC.

The following result is essentially due to Spielberg.  

\begin{lem}(Spielberg) Let $\C$ be a LCSC and let $\alpha, \beta \in \C$ with $\alpha \C \cap \beta \C = \gamma \C$. Then 
\[
\tau_{\beta}^{-1} \tau_{\alpha} = \tau_{\gamma^{\beta}} \tau_{\gamma^{\alpha}}^{-1}.
\]
where $\gamma^{\beta} = \tau_{\beta}^{-1}(\gamma)$ and $\gamma^{\alpha} = \tau_{\alpha}^{-1}(\gamma)$.
Moreover, the map $\tau_{\gamma^{\beta}} \tau_{\gamma^{\alpha}}^{-1}$ does not depend on the choice of $\gamma$.
\end{lem}

The argument for the first equality follows the proof of Lemma~3.3 in \cite{spielberg}.  
For any $\delta$ with $\alpha \C \cap \beta \C = \delta \C$,  since $ \tau_{\delta^{\beta}} \tau_{\delta^{\alpha}}^{-1}$ also equals $\tau_{\beta}^{-1} \tau_{\alpha}$, the map does not depend on the choice of $\gamma$.




We can use the lemma to give a nice description of the inverse semigroup of a singly aligned category that is reminiscent of the definition of a graph inverse semigroup.

\begin{thm}\label{thm:singlyalignedproducts} Let $\C$ be a singly aligned LCSC. Then 
\[
\ZM(\C) = \{ \tau_{\alpha}\tau_{\beta}^{-1} : s(\alpha) = s(\beta) \} \cup \{0\}.
\]
Moreover we have  
\[
\tau_{x}\tau_{\beta}^{-1}\tau_{\alpha}\tau_{y}^{-1} = \begin{cases}
		\tau_{x \gamma^{\beta}} \tau_{y \gamma^{\alpha}}^{-1} & \text{ if } \alpha \C \cap \beta \C = \gamma \C \\
		0 & \text{ otherwise}.
	\end{cases}
\]
\end{thm}
\begin{proof}

By definition, $\ZM(\C)$ consists of partial bijections 
\[
\tau_{\alpha_1}^{-1} \tau_{\beta_1} \tau_{\alpha_2}^{-1} \tau_{\beta_2} \dots \tau_{\alpha_n}^{-1} \tau_{\beta_n}
\]
where $\zeta = (\alpha_1, \beta_1, \dots, \alpha_n, \beta_n)$ is a zigzag. By repeated application of the above lemma, one can put an element of $\ZM(\C)$ in the required form. The formula for the product 
also follows from the lemma.
\end{proof}

We recall that $\gamma \C= \delta \C$ holds if and only if $\gamma = \delta \lambda$ for some invertible $\lambda \in \C$, as observed on \cite[page~5]{BKQS}.

\begin{prop}\label{prop:singlyalignedequality} Consider nonzero $\tau_{\alpha}\tau_{\beta}^{-1}$ and $\tau_{\gamma}\tau_{\sigma}^{-1}$ in $\ZM(\C)$. Then $\tau_{\alpha}\tau_{\beta}^{-1} = \tau_{\gamma}\tau_{\sigma}^{-1}$ if and only if there exists an invertible $\lambda$ in $\C$ such that $\beta = \delta \lambda$ and $\alpha = \gamma \lambda$.
\end{prop}

\begin{proof} Suppose that $\tau_{\alpha}\tau_{\beta}^{-1} = \tau_{\gamma}\tau_{\sigma}^{-1}$. Then $\beta \C = \sigma \C$ and hence there exists invertible $\lambda$ such that $\beta = \sigma \lambda$. Moreover 
\[
\alpha = \tau_{\alpha}\tau_{\beta}^{-1}(\beta) = \tau_{\gamma}\tau_{\sigma}^{-1}(\beta) = \gamma \lambda.
\]

Conversely, suppose that $\beta = \sigma \lambda$ and $\alpha = \gamma \lambda$ for some invertible $\lambda$ in $\C$. It follows that $\tau_{\alpha}\tau_{\beta}^{-1}$ and $\tau_{\gamma}\tau_{\sigma}^{-1}$ have the same domain $\beta \C$. For $x$ in $\beta \C$, write $x = \beta y = \sigma \lambda y$. Then $\tau_{\alpha}\tau_{\beta}^{-1}(x) = \alpha y = \gamma \lambda y = \tau_{\gamma}\tau_{\sigma}^{-1}(x)$.
\end{proof}

Jones and Lawson also considered singly aligned left cancellative categories in \cite{LawsonGraph}, referring to them as \emph{Leech categories} because of their connection with earlier work of Leech on constructing inverse monoids from small categories~\cite{Leech}.  
Jones and Lawson define an inverse semigroup $S(\C)$ of a singly aligned LCSC $\C$, which we now describe. Define $\equiv$ on the set $U = \{ (\alpha, \beta) \in \C : s(\alpha) = s(\beta)\}$ where
\[
(\alpha, \beta) \equiv (\alpha', \beta') 
\]
if and only if $\alpha = \alpha' \lambda$ and $\beta = \beta' \lambda$ for some invertible $\lambda$ in $\C$. Then $\equiv$ is an equivalence relation, and the equivalence class of $(\alpha, \beta)$ is denoted by $[\alpha, \beta]$. Now we let
\[
	S(\C) = \{ [\alpha,\beta] : s(\alpha) = s(\beta) \} \cup \{0\},
\]
with multiplication given by 
\[
[\alpha, \beta][\alpha', \beta'] = [\alpha x, \gamma y] \text{ if } \beta \C \cap \alpha' \C = \gamma \C \text{ and } \gamma = \beta x = \alpha' y
\]
and $[\alpha, \beta][\alpha', \beta'] = 0$ otherwise. The following proposition is now easily verified using Theorem \ref{thm:singlyalignedproducts} and Propostion \ref{prop:singlyalignedequality}.

\begin{prop} Let $\C$ be a singly aligned LCSC. Then $\ZM(\C)$ is isomorphic to $S(\C)$.
\end{prop}

Many of the following properties can now be gleaned from various places in the literature where $S(\C)$ appears or quickly verified by the reader. 
We include a proof of the last property.

\begin{prop} Let $\C$ be a singly aligned LCSC.
\begin{enumerate}

\item $E(\ZM(\C)) = \{ \tau_{\alpha}\tau_{\alpha}^{-1} : \alpha \in \C \} \cup \{0\}$.
\item $\tau_{\alpha} \tau_{\beta}^{-1} \leq \tau_{\gamma} \tau_{\sigma}^{-1}$ if and only if there exists $\mu \in \C$ such that $\alpha = \gamma \mu$ and $\beta = \sigma \mu$.
\item $\ZM(\C)$ is $0$-E-unitary if and only if $\C$ is right cancellative.
\item $\tau_{\alpha} \tau_{\beta}^{-1} \mcL \tau_{\gamma} \tau_{\sigma}^{-1}$ if and only if $\beta \sim \sigma$.
\item $\tau_{\alpha} \tau_{\beta}^{-1} \mcR \tau_{\gamma} \tau_{\sigma}^{-1}$ if and only if $\alpha \sim \gamma$.
\item $\tau_{\alpha} \tau_{\beta}^{-1} \mcH \tau_{\gamma} \tau_{\sigma}^{-1}$ if and only if $\beta \sim \sigma$ and $\alpha \sim \gamma$.
\item $\tau_{\alpha} \tau_{\beta}^{-1} \mcmu \tau_{\gamma} \tau_{\sigma}^{-1}$ if and only if $\beta \sim \sigma$ and $\alpha y^{\beta} \sim \gamma y^{\sigma}$ for all $y \in \beta \C = \sigma \C$.

\end{enumerate}
\end{prop}
\begin{proof}

(7) We will write $s = \tau_{\alpha} \tau_{\beta}^{-1}$ and $t =  \tau_{\gamma} \tau_{\sigma}^{-1}$. First suppose that $s \mcmu t$. Since $\mu \subseteq \mathcal{H}$ it follows that $\beta \sim \sigma$. Let $y \in \beta \C$. Then $s \tau_{y} \tau_{y}^{-1} s^* = \tau_{\alpha y^{\beta}} \tau_{\alpha y^{\beta}}^{-1}$ and $t \tau_{y} \tau_{y}^{-1} t^* = \tau_{\gamma y^{\sigma}} \tau_{\gamma y^{\sigma}}^{-1}$. Then $\alpha y^{\beta} \sim \gamma y^{\sigma}$ by Proposition \ref{prop:singlyalignedequality}. 

For the converse, fix an idempotent $\tau_{\mu}\tau_{\mu}^{-1}$. If $\mu \C \cap \beta \C = \emptyset,$ then $\mu \C \cap \sigma \C = \emptyset$ and $s \tau_{\mu}\tau_{\mu}^{-1} s^* = 0 = t \tau_{\mu}\tau_{\mu}^{-1} t^*$. Otherwise, $\beta \C \cap \mu \C = y \C = \sigma \C \cap \mu \C$ for some $y$ in $\C$.
It follows that
\begin{align*}
\tau_{\alpha} \tau_{\beta}^{-1} \tau_{\mu}\tau_{\mu}^{-1} \tau_{\beta} \tau_{\alpha}^{-1} &= \tau_{\alpha y^{\beta}}^{-1} \tau_{y} \tau_{\beta}^{-1} \tau_{\alpha} \\
             &= \tau_{\alpha y^{\beta}}\tau_{\alpha y^{\beta}}^{-1} \\
             &= \tau_{\gamma y^{\sigma}}\tau_{\gamma y^{\sigma}}^{-1} \\
             &= \tau_{\gamma} \tau_{\sigma}^{-1} \tau_{\mu}\tau_{\mu}^{-1} \tau_{\sigma} \tau_{\gamma}^{-1}.
\end{align*}
Thus $s \mcmu t$.
\end{proof}

\section{The path category of an inverse semigroup admitting unique maximal idempotents}

Say that an inverse semigroup $S$ \emph{admits unique maximal idempotents} if for every nonzero idempotent, there exists a unique maximal idempotent above it.
Thus, we have a map that sends each nonzero idempotent $e$ to the (unique) maximal idempotent $e^{\circ}$ such that $e \leq e^{\circ}$. 
In this section we introduce the path category of such an inverse semigroup. 
We aim to characterize the semigroups $S$ for which $S$ is isomorphic to $\ZM(\C)$ for some LCSC $\C$. 
This goal is inspired by the characterization of graph inverse semigroups in \cite{LawsonGraph}. 

Jones and Lawson define $S$ to be a \emph{Perrot inverse semigroup} if it satisfies the following properties:

\begin{enumerate}

\item[(P1)] The semilattice of idempotents is unambiguous.

\item[(P2)] For each nonzero idempotent $e$ there are finitely many idempotents above $e$ in the natural partial order.

\item[(P3)] $S$ admits unique maximal idempotents.

\item[(P4)] Each nonzero $\mcD$-class of $S$ contains a maximal idempotent.

\end{enumerate}

Also, $S$ is a \emph{proper} Perrot inverse semigroup if it satisfies the above properties and, in addition, there is a unique maximal idempotent in any nonzero $\mcD$-class. Jones and Lawson obtained the following characterization:

\begin{thm}[\cite{LawsonGraph}] The graph inverse semigroups are precisely the combinatorial proper Perrot semigroups.
\end{thm}

The third condition is arguably the most important as it allows one to define the path category of the inverse semigroup.

\begin{defn} Let $S$ be an inverse semigroup that admits unique maximal idempotents. We say that $s \in S$ is a \emph{path} if $s^*s$ is maximal. We denote by $P(S)$ (or sometimes just $P$) the set of paths in $S$.  
\end{defn}

\begin{exmp}
To motivate the previous definition, we show that, for a graph inverse semigroup, the paths of the inverse semigroup correspond to the paths of the original graph.
Letting $\Lambda$ by a directed graph, one can quickly check that 
\[ E(S_\Lambda) = \{ (\alpha, \alpha) : \alpha \in \Lambda^* \} \cup \{0\}, \]
and that $(\alpha, \alpha) \leq (\beta, \beta)$ if and only if $\alpha = \beta \gamma$ for some path $\gamma$. It follows that each nonzero idempotent $(\alpha, \alpha)$ lies under the unique maximal idempotent $(s_\alpha, s_\alpha)$. For $s = (\alpha, \beta)$ in $S_{\Gamma}$, note that $s^*s = (\beta, \beta)$. Thus there is a correspondence between $\Gamma^*$ and the set $P$ of paths in $S_{\Gamma}$ since 
\[
P = \{ (\alpha, s_{\alpha}) : \alpha \in \Gamma^*\}.
\]
\end{exmp}

\begin{defn}
We define a left cancellative category $\C = \C(S)$, which we call the \emph{path category of $S$}, for an inverse semigroup $S$ that admits unique maximal idempotents.
Let
\[
	\C = \C(S) := \left\{ ((ss^*)^{\circ},s) \in E \times S : s \text{ is a path}\right\}.
\]
It is easy to verify that the objects correspond to the maximal idempotents of $S$ and the morphisms correspond to $P(S)$, the paths of $S$. 
Given $\alpha = (e,s)$ in $\C(S)$ we write $s(\alpha) = (s^*s, s^*s)$ and $r(\alpha) = (e,e)$. Then $(e,s)$ and $(f,t)$ in $\C$ are composable when $s^*s = f$, and the product is defined to be 
\[
(e,s)(f,t) = (e, st).
\]
\end{defn}

We omit the routine proof of the following proposition.

\begin{prop} Let $S$ be an inverse semigroup admitting unique maximal idempotents. The path category $\C$ of $S$ is a left cancellative category.
\end{prop}

We note that, if $S$ is any inverse semigroup satisfying (P3) and (P4), then $\C$ is singly aligned and $S$ is isomorphic to $\ZM(\C) = S(\C)$. This follows from \cite[Theorem 2.8]{LawsonGraph}. 
We give the following example of an inverse semigroup that satisfies (P3) but not (P4).

\begin{exmp}\label{ex:notgraph} Consider the inverse subsemigroup $S$ of $I_{11}$ generated by the partial bijections
\[
e = \bigl(\begin{smallmatrix}
    7 & 8 & 9 & 10\\
    7 & 8 & 9 & 10
    \end{smallmatrix}\bigr),\; a = \bigl(\begin{smallmatrix}
    1 & 2 & 3\\
    7 & 8 & 9
    \end{smallmatrix}\bigr),\; b = \bigl(\begin{smallmatrix}
    4 & 5 & 6\\
    7 & 8 & 10
    \end{smallmatrix}\bigr), \text{ and } c = \bigl(\begin{smallmatrix}
    11 \\
   7 
    \end{smallmatrix}\bigr).
\]

We give the eggbox diagram of each non-zero $\mcD$-class of $S$. The maximal idempotents have been labelled with the $\symking$ symbol.

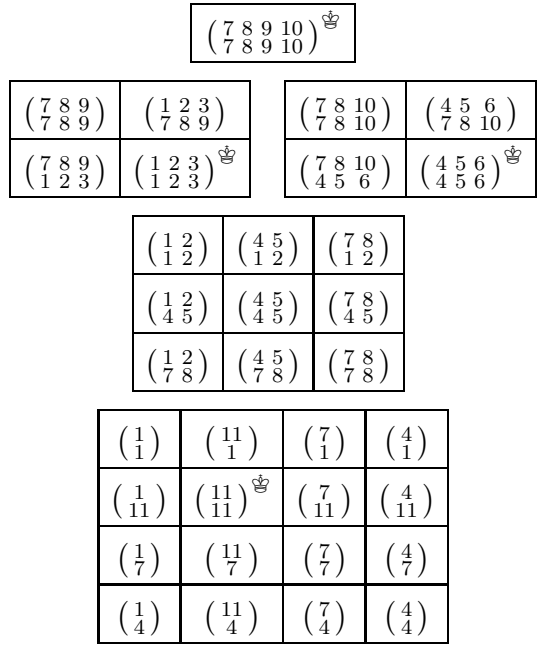
\vspace{2mm}
\begin{center}
	\renewcommand{\arraystretch}{1.667}
    \noindent \begin{tabular}{|c|} \hline
           $\bigl(\begin{smallmatrix}
  7 & 8 & 9 & 10\\
  7 & 8 & 9 & 10
\end{smallmatrix}\bigr)^{\symking}$ \\ \hline
    \end{tabular}

\vspace{2mm}

    \begin{tabular}{cc}
    \begin{tabular}{| @{\hspace{1.5mm}}c@{\hspace{1.5mm}} | @{\hspace{1.5mm}}c@{\hspace{1.5mm}} |} \hline
        $\bigl(\begin{smallmatrix}
  7 & 8 & 9 \\
  7 & 8 & 9 
\end{smallmatrix}\bigr)$ & $\bigl(\begin{smallmatrix}
  1 & 2 & 3\\
  7 & 8 & 9
\end{smallmatrix}\bigr)$ \\ \hline
        $\bigl(\begin{smallmatrix}
  7 & 8 & 9\\
  1 & 2 & 3
\end{smallmatrix}\bigr)$ &     $\bigl(\begin{smallmatrix} 
    1 & 2 & 3\\ 
    1 & 2 & 3
    \end{smallmatrix}\bigr)^{\symking}$ \\ \hline
    \end{tabular} 
    
\hspace{2mm}

    \begin{tabular}{| @{\hspace{1.5mm}}c@{\hspace{1.5mm}} | @{\hspace{1.5mm}}c@{\hspace{1.5mm}} |} \hline
        $\bigl(\begin{smallmatrix}
  7 & 8 & 10\\
  7 & 8 & 10 
\end{smallmatrix}\bigr)$ & $\bigl(\begin{smallmatrix}
  4 & 5 & 6 \\
  7 & 8 & 10
\end{smallmatrix}\bigr)$ \\ \hline
        $\bigl(\begin{smallmatrix}
  7 & 8 & 10\\
  4 & 5 & 6
\end{smallmatrix}\bigr)$ &     $\bigl(\begin{smallmatrix} 
    4 & 5 & 6\\ 
    4 & 5 & 6 
    \end{smallmatrix}\bigr)^{\symking}$ \\ \hline
    \end{tabular}
    \end{tabular}

\vspace{2mm}

    \begin{tabular}{| @{\hspace{1.5mm}}c@{\hspace{1.5mm}} | @{\hspace{1.5mm}}c@{\hspace{1.5mm}} | @{\hspace{1.5mm}}c@{\hspace{1.5mm}} |} \hline
            $\bigl(\begin{smallmatrix}
      1 & 2 \\
      1 & 2 
    \end{smallmatrix}\bigr)$ & $\bigl(\begin{smallmatrix}
      4 & 5\\
      1 & 2
    \end{smallmatrix}\bigr)$ &
    $\bigl(\begin{smallmatrix}
      7 & 8\\
      1 & 2
    \end{smallmatrix}\bigr)$ \\ \hline
    $\bigl(\begin{smallmatrix}
      1 & 2\\
      4 & 5
    \end{smallmatrix}\bigr)$ &     $\bigl(\begin{smallmatrix} 
        4 & 5\\ 
        4 & 5
        \end{smallmatrix}\bigr)$ &
    $\bigl(\begin{smallmatrix}
      7 & 8\\
      4 & 5
    \end{smallmatrix}\bigr)$\\ \hline
    $\bigl(\begin{smallmatrix}
      1 & 2\\
      7 & 8
    \end{smallmatrix}\bigr)$ &     $\bigl(\begin{smallmatrix} 
        4 & 5\\ 
        7 & 8
        \end{smallmatrix}\bigr)$ &
    $\bigl(\begin{smallmatrix}
      7 & 8\\
      7 & 8
    \end{smallmatrix}\bigr)$\\ \hline
        \end{tabular}
\vspace{2mm}

    \begin{tabular}{| @{\hspace{1.5mm}}c@{\hspace{1.5mm}} | @{\hspace{1.5mm}}c@{\hspace{1.5mm}} | @{\hspace{1.5mm}}c@{\hspace{1.5mm}} | @{\hspace{1.5mm}}c@{\hspace{1.5mm}} |} \hline
    $\bigl(\begin{smallmatrix}
      1\\1
    \end{smallmatrix}\bigr)$ &     $\bigl(\begin{smallmatrix} 
        11\\1
        \end{smallmatrix}\bigr)$ &
    $\bigl(\begin{smallmatrix}
      7\\1
    \end{smallmatrix}\bigr)$
    &
    $\bigl(\begin{smallmatrix}
      4 \\1 
    \end{smallmatrix}\bigr)$ \\ \hline
    $\bigl(\begin{smallmatrix}
      1\\11
    \end{smallmatrix}\bigr)$ &     $\bigl(\begin{smallmatrix} 
        11\\11
        \end{smallmatrix}\bigr)^{\symking}$ &
    $\bigl(\begin{smallmatrix}
      7\\11
    \end{smallmatrix}\bigr)$
    &
    $\bigl(\begin{smallmatrix}
      4 \\11 
    \end{smallmatrix}\bigr)$ \\ \hline
    $\bigl(\begin{smallmatrix}
      1\\7
    \end{smallmatrix}\bigr)$ &     $\bigl(\begin{smallmatrix} 
        11\\7
        \end{smallmatrix}\bigr)$ &
    $\bigl(\begin{smallmatrix}
      7\\7
    \end{smallmatrix}\bigr)$
    &
    $\bigl(\begin{smallmatrix}
      4\\7
    \end{smallmatrix}\bigr)$ \\ \hline
    $\bigl(\begin{smallmatrix}
      1\\ 4
    \end{smallmatrix}\bigr)$ &     $\bigl(\begin{smallmatrix} 
        11\\ 4
        \end{smallmatrix}\bigr)$ &
    $\bigl(\begin{smallmatrix}
      7\\ 4
    \end{smallmatrix}\bigr)$
    &
    $\bigl(\begin{smallmatrix}
      4 \\ 4 
    \end{smallmatrix}\bigr)$ \\ \hline
        \end{tabular}
\end{center}
\vspace{3mm}

By definition, the set of paths in $S$ is the union of the $\mcL$-classes of maximal idempotents. The non-idempotent paths in this example are: $a, b, c, x,$ and $y$ where $x = \bigl(\begin{smallmatrix}
    11 \\
   1 
\end{smallmatrix}\bigr)$ and $y = \bigl(\begin{smallmatrix}
    11 \\
   4 
\end{smallmatrix}\bigr)$. 

Note that there is one nonzero $\mcD$-class that does not contain a maximal idempotent. Thus $S$ is not a graph inverse semigroup. However, the path category of $S$ is isomorphic to the path category of the following directed graph with identifications $c = ax = by$.
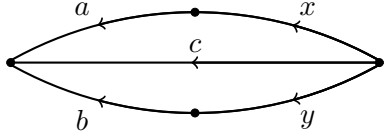
\vspace{3mm}

\begin{center}
\begin{tikzpicture}

\draw [-latex,thick][black] [-] (7,0) arc [radius=5, start angle=60, end angle= 120] node[pos=0.4, inner sep=5.0pt, anchor=north east] {}; 
\node [above] at (3,.5) {$a$};

\draw [-latex,thick][black] [->] (7,0) arc [radius=5, start angle=60, end angle= 105] node[pos=0.4, inner sep=5.0pt, anchor=north east] {};
\node [above] at (6,.5) {$x$};

\draw [-latex,thick][black] [->] (7,0) arc [radius=5, start angle=60, end angle= 75] node[pos=0.4, inner sep=5.0pt, anchor=north east] {};

\draw [-latex,thick][black] [-] (7,0) arc [radius=5, start angle=-60, end angle= -120] node[pos=0.4, inner sep=5.0pt, anchor=north east] {};

\draw [-latex,thick][black] [->] (7,0) arc [radius=5, start angle=-60, end angle= -105] node[pos=0.4, inner sep=5.0pt, anchor=north east] {};
\node [below] at (3,-.5) {$b$};

\draw [-latex,thick][black] [->] (7,0) arc [radius=5, start angle=-60, end angle= -75] node[pos=0.4, inner sep=5.0pt, anchor=north east] {};
\node [below] at (6,-.5) {$y$};

\draw [-latex,thick][black][-] (7, 0) -- (2.0,0); 
\draw [-latex,thick][black][->] (7, 0) -- (4.45,0); 
\node [above] at (4.5,0) {$c$};

\foreach \x in {4.5} \filldraw [black] (\x,0.67) circle (1.5pt);
\foreach \x in {4.5} \filldraw [black] (\x,-.67) circle (1.5pt);
\foreach \x in {2.05} \filldraw [black] (\x,0) circle (1.5pt);
\foreach \x in {6.95} \filldraw [black] (\x,0) circle (1.5pt);

\end{tikzpicture}
\end{center}

The paths $a,b,$ and $c$ share a range object in the diagram because 
\[
(cc^*)^{\circ} = \bigl(\begin{smallmatrix}
    7 \\
    7 
\end{smallmatrix}\bigr)^{\circ} = \bigl(\begin{smallmatrix}
  7 & 8 & 9 & 10\\
  7 & 8 & 9 & 10
\end{smallmatrix}\bigr) = (aa^*)^{\circ} = (bb^*)^{\circ}.
\]
In Example~\ref{ex:notzigzag}, we show that $S$ is not a zigzag inverse semigroup either, as it turns out that it does not satisfy Condition~(Z3) mentioned in the introduction.
\end{exmp}

\section{A Characterization of Zigzag Inverse Semigroups} 

In this section we characterize inverse semigroups that are isomorphic to $\ZM(\C)$ for some LCSC $\C$. We need a few more definitions.

\begin{defn}\label{def:dompath} Let $S$ be an inverse semigroup that admits unique maximal idempotents. Given $s \in S$ we say that a path $x \in P(S)$ is a \emph{domain path of $s$} if $xx^* \leq s^*s$. We denote by $P_s$ the set of domain paths of $s$.
\end{defn}

A semigroup $S$ is said to be \emph{right reductive} if for any $s,t$ in $S$, 
\[
	sx = tx \text{ for all } x \in S \text{ implies that } s = t.
\]
One can quickly prove that inverse semigroups are right reductive. We introduce a more restrictive condition involving domain paths.

\begin{defn} Let $S$ be an inverse semigroup that admits unique maximal idempotents. We say that $S$ is \emph{right reductive on domain paths} if for any $s,t$ in $S$, the conditions:
\[
	P_s = P_t \quad \text{and} \quad sx = tx \;\text{ for all } x \in P_s
\]
imply that $s = t$.
\end{defn}

\begin{rem} The definition of right reductive on domain paths was motivated by work of Cherubini and Petrich on the inverse hull of a right cancellative semigroup \cite{CherubiniPetrich}. 
In Section 6 of \cite{CherubiniPetrich}, they show their inverse hull operation provides one direction of an equivalence between two categories.  
One of the categories is inverse monoids satisfying several conditions, including a condition similar to being right reductive on domain paths,
see Definition 6.2 of \cite{CherubiniPetrich}.
Their proofs are for monoids and do not generalize in a straightforward way to our context, as we have multiple maximal idempotents. 
%
\end{rem}

The next lemma motivates the terminology in Definition \ref{def:dompath}.

\begin{lem}\label{lem:Zpaths} Let $\C$ be a LCSC. Then $\ZM(\C)$ admits unique maximal idempotents, and the paths in $\ZM(\C)$ are precisely the elements of the form $\tau_{\alpha}$ for $\alpha \in \C$. Additionally, for any $\phi_{\zeta}$ in $\ZM(\C)$, 
\[
	P_{\phi_{\zeta}} = \{ \tau_{\alpha} : \alpha \in \dom(\phi_{\zeta})\}.
\]
\end{lem}
\begin{proof} The maximal idempotents in $\ZM(\C)$ are the idempotents $\tau_{u}$ for an object $u$ in $\C$. Also, if $\tau_{u} \neq \tau_{v}$ then $\tau_u \tau_v = 0$ since $u\C \cap v\C = \emptyset$. If $\phi_{\zeta} = \tau_{\alpha_1}^{-1} \tau_{\beta_1} \cdots \tau_{\alpha_n}^{-1} \tau_{\beta_n}$ is the map associated with a zigzag $\zeta$ and $u$ is an object, then 
\[\phi_{\zeta}^{-1} \phi_{\zeta} \tau_{u} = 
\begin{cases}
		\phi_{\zeta}^{-1} \phi_{\zeta} & \text{ if } \tau_u = \tau_{s(\beta_n)}  \\
		0 & \text{ otherwise}.
	\end{cases}
\]
Therefore each nonzero idempotent of $\ZM(\C)$ lies under a unique maximal idempotent. 
Notice that for each $\alpha \in \C$, $\tau_{\alpha}^{-1} \tau_{\alpha} = \tau_{s(\alpha)}$. Thus $\tau_{\alpha}$ is a path. Suppose that $\phi_{\zeta} = \tau_{\alpha_1}^{-1} \tau_{\beta_1} \cdots \tau_{\alpha_n}^{-1} \tau_{\beta_n}$ is a path. Then $\phi_{\zeta}^{-1} \phi_{\zeta} = \tau_{s(\beta_n)}$ which implies that $s(\beta_n)$ is in the domain of $\tau_{\alpha_n}^{-1} \tau_{\beta_n}$. So $\beta_n \in \alpha_n \C$. Write $\beta_n = \alpha_n \mu_n$ for some $\mu_n$ in $\C$. Thus $\tau_{\alpha_n}^{-1} \tau_{\beta_n} = \tau_{\mu_n}$. Now we may write $\phi_{\zeta}$ as a product of $n-1$ pairs of maps:
\[
\phi_{\zeta} = \tau_{\alpha_1}^{-1} \tau_{\beta_1} \cdots \tau_{\alpha_{n-1}}^{-1} \tau_{\beta_{n-1} \mu_n}.
\]
Repeating the argument we find $\mu_i \in \C$ for $i = 1, \dots, n$ such that 
\[
\phi_{\zeta} = \tau_{\mu_1 \mu_2 \cdots \mu_n}.
\]

Now that we have described the paths, we turn to the set $P_{\phi_{\zeta}}$ of domain paths of a zigzag map $\phi_{\zeta}$. Let $\alpha \in \dom(\phi_{\zeta})$. Note that $\alpha \beta \in \dom(\phi_{\zeta})$ for any $\beta \in s(\alpha)\C$ and
\[
\phi_{\zeta}^{-1}\phi_{\zeta}(\alpha \beta) = \alpha \beta = \tau_{\alpha} \tau_{\alpha}^{-1}(\alpha \beta). 
\]
Since $\dom(\tau_{\alpha} \tau_{\alpha}^{-1}) = \alpha\C$, we have $\tau_{\alpha} \tau_{\alpha}^{-1} \leq \phi_{\zeta}^{-1}\phi_{\zeta}$. Conversely, if $\tau_{\alpha} \tau_{\alpha}^{-1} \leq \phi_{\zeta}^{-1}\phi_{\zeta}$, then $\alpha \in \dom(\phi_{\zeta})$. Therefore we have shown that
\[
P_{\phi_{\zeta}} = \{ \tau_{\alpha} : \alpha \in \dom(\phi_{\zeta})\}.
\]
\end{proof}

The three properties that characterize zigzag inverse semigroups are:
\begin{enumerate}

\item[(Z1)] $S$ admits unique maximal idempotents.

\item[(Z2)] The set of paths $P(S)$ generate $S$.

\item[(Z3)] $S$ is right reductive on domain paths.

\end{enumerate}

We will first show that $\ZM(\C)$ satisfies each of the above properties.

\begin{thm} Let $\C$ be a LCSC. Then $\ZM(\C)$ satisfies (Z1), (Z2), and (Z3).
\end{thm}
\begin{proof} We proved in Lemma \ref{lem:Zpaths} that $\ZM(\C)$ admits unique maximal idempotents. Also, since $\tau_{\alpha}$ is a path for each $\alpha \in \C$, it follows by definition that $\ZM(\C)$ satisfies (Z2).

Before proving that $\ZM(\C)$ satisfies (Z3) we claim the following: if $\phi_{\zeta} \in \ZM(\C)$ and $\alpha \in \dom(\phi_{\zeta})$, then $\phi_{\zeta} \tau_{\alpha} = \tau_{\phi_{\zeta}(\alpha)}$. Write $\phi_{\zeta} = \tau_{\alpha_1}^{-1} \tau_{\beta_1} \cdots \tau_{\alpha_n}^{-1} \tau_{\beta_n}$. Then $\alpha \in \dom(\tau_{\alpha_n}^{-1}\tau_{\beta_n})$ and $\tau_{\alpha_n}^{-1}\tau_{\beta_n}(\alpha) = \gamma$ for some $\gamma$ such that $\alpha_n \gamma = \beta_n \alpha$. Thus $\dom(\tau_{\alpha_n}^{-1}\tau_{\beta_n}\tau_{\alpha}) = s(\alpha) \C = \dom(\gamma \C)$ and for any $\mu \in s(\alpha) \C,$
\[
	\tau_{\alpha_n}^{-1}\tau_{\beta_n}\tau_{\alpha}(\mu) = \tau_{\alpha_n}^{-1}(\beta_n \alpha)\mu = \tau_{\gamma}(\mu).
\]
Repeating this argument for $\tau_{\alpha_i}^{-1}\tau_{\beta_i}$ for each $i$ we conclude that $\dom(\phi_{\zeta} \tau_{\alpha}) = s(\alpha) \C = \dom(\tau_{\phi_{\zeta}(\alpha)})$ and for any $\mu \in s(\alpha)\C$, $\phi_{\zeta}(\mu) = \tau_{\alpha}(\mu)$. This proves the claim.

Finally, suppose that $\phi_{\zeta_1}$ and $\phi_{\zeta_2}$ are zigzag maps such that $P_{\phi_{\zeta_1}} = P_{\phi_{\zeta_2}}$ and $\phi_{\zeta_1} \tau_{\alpha} = \phi_{\zeta_2} \tau_{\alpha}$ for all $\tau_{\alpha} \in P_{\phi_{\zeta_1}}$. Then, by Lemma \ref{lem:Zpaths}, $\dom(\phi_{\zeta_1}) = \dom(\phi_{\zeta_2})$ and for $\alpha \in \dom(\phi_{\zeta_1})$, $\tau_{\phi_{\zeta_1}(\alpha)} = \tau_{\phi_{\zeta_2}(\alpha)}$. Thus $\phi_{\zeta_1} = \phi_{\zeta_2}$, and $\ZM(\C)$ is right reductive on paths.
\end{proof}

\begin{exmp}\label{ex:notzigzag} We return briefly to Example \ref{ex:notgraph}. One can show that $S$ satisfies (Z1) and (Z2). However, $S$ does not satisfy (Z3) since for $s = b^*a$ and $t = yx^*$ we have
\[
	P_s = \{ \left(\begin{smallmatrix}
    11 \\
    1 
\end{smallmatrix}\right), \left(\begin{smallmatrix}
    1 \\
    1 
\end{smallmatrix}\right)\} = P_t
\]
with $sx = tx$ for each $x \in P_s$, yet $s \neq t$. It follows that $S$ is not a zigzag inverse semigroup. 
At the end of Example~\ref{ex:notgraph}, we observed that the path category of $S$, $\C$, is isomorphic to the path category of a directed graph with certain identifications. There is a natural relationship between $\ZM(\C)$ and the original inverse semigroup $S$, which we present in the next remark.
\end{exmp}

\begin{rem}
For any inverse semigroup $T$ with $0$ satisfying (Z1) and (Z2), there is a congruence on $T$ given by identifying $s$ with $t$ if $P_s = P_t$ and $sx = tx$ for all $x \in P_s$. 
One can then show that $\ZM(\C(T))$ is the quotient of $T$ by this congruence.
\end{rem}

To prove our main theorem, we will show that if $S$ satisfies (Z1), (Z2), and (Z3), then $S$ is isomorphic to  $\ZM(\C)$ where $\C$ is the path category of $S$. 
To simplify the notation, for a path $a$ in $S$ we let $t(a) = ((aa^*)^{\circ},a)$ in $\C$ and we let $\tau_{a}$ denote the map $\tau_{t(a)}$ in $\ZM(\C)$.

\begin{lem}\label{lem:Zintersect} Let $S$ be an inverse semigroup satisfying (Z1) and let $\C$ be the path category of $S$. For paths $a$ and $b$ in $S$ we have
\[
	t(a)\C \cap t(b)\C = \{t(z) : z^*z \text{ is maximal and } zz^* \leq aa^*bb^* \}.
\]
Also, 
\[
\dom(\tau_{a}^{-1} \tau_{b}) = \{t(x): x = b^*z, z^*z \text{ is maximal, and } zz^* \leq aa^*bb^* \}.
\]
For such $x$, $\tau_{a}^{-1} \tau_{b}(t(x)) = t(a^*bx)$.
\end{lem}
\begin{proof}
For the first claim, suppose that $t(z) \in t(a)\C \cap t(b)\C$. Then $z^*z$ is maximal since $z$ is a path and $z = ax = by$ for some paths $x$ and $y$. 
Since $t(a)$ and $t(x)$ are composable, we have $xx^* \leq a^*a$ and hence $zz^* = axx^*a^* \leq aa^*$. Similarly $zz^* \leq bb^*$ and so $zz^* \leq aa^*bb^*$. Next suppose that $z$ is a path with $zz^* \leq aa^*bb^*$. Let $x = a^*z$ and $y = b^*z$. Note that $x^*x = z^* aa^* z = z^*z$ and $xx^* \leq a^*a$. Thus $t(x) \in t(a)\C$ and $t(z) = t(a)t(x)$. Similarly $t(z) = t(b)t(y)$.

Next, notice that $\dom(\tau_{a}^{-1} \tau_{b})$ consists of $t(x) \in t(b^*b)\C$ such that $t(bx) \in t(a)\C \cap t(b)\C$. So for $x$ in $\dom(\tau_{a}^{-1} \tau_{b})$ and $z = bx$ we have $z^*z = x^*b^*bx = x^* x$ is maximal and $zz^* \leq aa^*bb^*$. Conversely, if $x = b^*z$ where $z^*z$ is maximal and $zz^* \leq a^*ab^*b$ then $t(bx) = t(z) \in t(a)\C \cap t(b)\C$. Finally, we have
\[
\tau_{a}^{-1} \tau_{b}(t(x)) = \tau_{a}^{-1}(t(bb^*z)) = \tau_{a}^{-1}(t(aa^*bb^*z)) = t(a^*bb^*z) = t(a^*bx).
\]
\end{proof}

\begin{lem}\label{lem:Zdomain} Let $S$ be an inverse semigroup satisfying (Z1) and let $\C$ be the path category of $S$. If $\phi = \tau_{a_1}^{-1} \tau_{b_1} \cdots \tau_{a_n}^{-1} \tau_{b_n} \in \ZM(\C)$, then 
\[
P_{a_1^* b_1 \cdots a_n^* b_n} = \{x \in P(S) : t(x) \in \dom(\phi)\}.
\]
\end{lem}
\begin{proof}
Since $\phi \in \ZM(\C)$, $s(t(b_i)) = s(t(a_{i+1}))$ for $i = 1, \dots, n-1$. Using the previous lemma, the domain of $\phi$ consists of $t(b_n^*z_n)$ such that $z_n$ is a path, $z_n z_n^* \leq a_n a_n^* b_n b_n^*,$ and 
\[
a_n^* z_n = \tau_{a_n}^{-1} \tau_{b_n} (b_n^* z_n) \in \dom(\tau_{a_1}^{-1} \tau_{b_1} \cdots \tau_{a_{n-1}}^{-1} \tau_{b_{n-1}}).
\] 
Then $a_n^* z_n = b_{n-1}^* z_{n-1}$ for some $z_{n-1}$ where $z_{n-1} z_{n-1}^* \leq a_{n-1} a_{n-1}^* b_{n-1} b_{n-1}^*$. Notice that $z_{n-1} = b_{n-1} a_n^* z_n$ and, since $s(t(b_{n-1})) = s(t(a_n))$, we see that
\begin{align*}
z_{n-1}^* z_{n-1} 	&= z_n^* a_n b_{n-1}^* b_{n-1} a_n^* z_n \\
					&= z_n^* a_n a_n^* z_n \\
					&= (a_n^* z_n)^* (a_n^* z_n).
\end{align*}
Thus $z_{n-1}$ is automatically a path. Continuing in this way, we can show inductively that the domain of $\phi$ consists of $t(b_n^*z_n)$ where $z_n \in P(S)$, $z_{n} z_{n}^* \leq a_{n} a_{n}^* b_{n} b_{n}^*$, and for $z_{k} = b_{k} a_{k+1}^* z_{k+1},\; z_{k} z_{k}^* \leq a_{k} a_{k}^* b_{k} b_{k}^*$ for $k = n-1, \dots, 1$.

We are now prepared to prove the lemma. Let $s_k = a_1^* b_1 \cdots a_k^* b_k$ for each $k = 1, \dots, n$. Suppose that $t(x) \in \dom(\phi)$. Choose $z_k$ for $k = n, \dots, 1$ as above. Then
\begin{align*}
 xx^* 	&= b_n^* z_n z_n^* b_n \\
 		&= b_n^* a_n (a_n^* z_n z_n^* a_n) a_n^* b_n \\
		&= b_n^* a_n (b_{n-1}^* z_{n-1} z_{n-1}^* b_{n-1}) a_n^* b_n \\
		&\;\;\vdots\\
		&= s_n^* z_1 z_1^* s_n \\
		&\leq s_n^* s_n.
\end{align*}  
Thus $x \in P_{s_n}$.
Conversely, suppose that $x$ is a domain path for $s_n$. Then 
\[
xx^* \leq  s_n^* s_n \leq b_n^* b_n
\]
and $x = b_n^* z_n$ where $z_n = b_n x$. Then $z_n \in P(S)$ and 
\begin{align*}
z_n z_n^* 	&= b_n x x^* b_n^* \\
			&\leq b_n s_n^* s_n b_n^* \\
			&= b_n b_n^* a_n s_{n-1}^* s_{n-1} a_n^* b_n b_n^* \\
			&\leq a_n a_n^* b_n b_n^*
\end{align*} 
For each $k$, let $z_k = b_k a_{k+1}^* z_{k+1}$. Then we have
\begin{samepage}
\begin{align*}
z_k z_k^* 	&= b_k a_{k+1}^* (z_{k+1} z_{k+1}^*) a_{k+1} b_k^* \\
		 	&= b_k a_{k+1}^* b_{k+1} a_{k+2}^* (z_{k+2} z_{k+2}^*) a_{k+2} b_{k+1}^* a_{k+1} b_k^* \\
			&\;\;\vdots\\
			&= b_k a_{k+1}^* \cdots b_{n-1} a_n^* (z_n z_n^*) a_n b_{n-1}^* \cdots a_{k+1} b_k^* \\
			&\leq b_k a_{k+1}^* \cdots b_{n-1} a_n^* (b_n s_n^* s_n b_n^*) a_n b_{n-1}^* \cdots a_{k+1} b_k^* \\
\intertext{and, letting $e_k = b_{k+1} a_{k+2}^* \cdots b_{n-1}a_{n}^*b_n b_n^* a_n b_{n-1}^* \cdots a_{k+2} b_{k+1}^*$, we rewrite the last line as}
			&= b_k (a_{k+1}^* e_k a_{k+1}) (s_k^* s_k)(a_{k+1}^* e_k a_{k+1}) b_k^* \\
			&\leq b_k (s_k^* s_k) b_k^*\\
			&\leq a_k a_k^* b_k b_k^*
\end{align*}
\end{samepage}
Therefore $t(x) \in \dom(\phi)$.
\end{proof}

\begin{prop}\label{prop:Szigzag} Let $S$ be an inverse semigroup satisfying (Z1). Then the set $P^0$ of paths in $S$ together with $0$ is a subsemigroup of $S$. Moreover, every nonzero element in the inverse semigroup generated by $P$ can be written in the form $a_1^* b_1 a_2^* b_2 \cdots a_n^* b_n$ where $(a_i a_i^*)^{\circ} = (b_i b_i^*)^{\circ}$ for $i = 1, \dots, n$ and $b_i^* b_i  = a_{i+1}^* a_{i+1}$ for $i = 1, \dots, n-1$.
\end{prop}
\begin{proof}
Let $a$ and $b$ be paths in $S$. By (Z1), $a^*abb^*$ lies under a unique maximal idempotent if it is nonzero. Thus 
\[
	a^*abb^* = \begin{cases}
			 	a^*a & \text{ if } (bb^*)^{\circ} = a^*a \\
				0 & \text{ otherwise}.
		 \end{cases}
\]

We conclude that: $ab \neq 0 \Leftrightarrow a^*abb^* \neq 0 \Leftrightarrow bb^* \leq a^*a$. So $ab \neq 0$ if and only if $ab$ is a path. Thus $P^0$ is a semigroup. Similarly, the product $(ab)^*$ is nonzero exactly when it is the inverse of a path. Thus any nonzero product in the inverse semigroup generated by $P$ can be reduced to one that alternates between paths and inverses of paths. After left or right multiplying by the correct maximal idempotent, one can assume the alternating product has the form $a_1^* b_1 a_2^* b_2 \cdots a_n^* b_n$. One can also assume $(a_i a_i^*)^{\circ} = (b_i b_i^*)^{\circ}$ for $i = 1, \dots, n$ and $b_i^* b_i  = a_{i+1}^* a_{i+1}$ for $i = 1, \dots, n-1$, because otherwise the product will be $0$. 
\end{proof}

\begin{thm}\label{thm:main} The zigzag inverse semigroups are precisely the inverse semigroups with $0$ satisfying (Z1), (Z2), and (Z3).
\end{thm}
\begin{proof} We have seen that any semigroup of the form $\ZM(\C)$ satisfies (Z1), (Z2), and (Z3). To complete the proof assume that $S$ is an inverse semigroup with $0$ satisfying (Z1), (Z2), and (Z3). Let $s \in S$. Suppose that there are two representations of $s$ as in Proposition \ref{prop:Szigzag}. Say $s = a_1^* b_1 \cdots a_n^* b_n = c_1^* d_1 \cdots c_m^* d_m$. Then 
\[ 
    P_{a_1^* b_1 \cdots a_n^* b_n} = \{x \in P(S) : xx^* \leq s^*s \} = P_{c_1^* d_1 \cdots c_n^* d_n}. 
\]
By Lemma \ref{lem:Zdomain}, the domain of $\tau_{a_1}^{-1}\tau_{b_1} \cdots \tau_{a_n}^{-1} \tau_{b_n}$ is equal to the domain of $\tau_{c_1}^{-1}\tau_{d_1} \cdots \tau_{c_n}^{-1} \tau_{d_n}$. It follows from the final assertion of Lemma \ref{lem:Zintersect} that for $t(x) \in \dom(\tau_{a_1}^{-1}\tau_{b_1} \cdots \tau_{a_n}^{-1} \tau_{b_n})$, 
\begin{align*}
\tau_{a_1}^{-1}\tau_{b_1} \cdots \tau_{a_n}^{-1} \tau_{b_n}(t(x)) &= t(a_1^* b_1 \cdots a_n^* b_n x) \\
 &= t(c_1^* d_1 \cdots c_m^* d_m x) \\
 &= \tau_{c_1}^{-1}\tau_{d_1} \cdots \tau_{c_n}^{-1} \tau_{d_n}(t(x)). \\
\end{align*}
Therefore $\tau_{a_1}^{-1}\tau_{b_1} \cdots \tau_{a_n}^{-1} \tau_{b_n} = \tau_{c_1}^{-1}\tau_{d_1} \cdots \tau_{c_n}^{-1} \tau_{d_n}$.
Thus there is a well-defined map $\theta : S \to \ZM(\C)$ given by 
\[
\theta(s) = \tau_{a_1}^{-1}\tau_{b_1} \cdots \tau_{a_n}^{-1} \tau_{b_n}.
\]
It is easily verified that $\theta$ is a homomorphism. Moreover, it follows from (Z3) that $\theta$ is injective. Finally, it follows from the definition of $\ZM(\C)$ and Proposition \ref{prop:Szigzag} that $\theta$ is surjective.
\end{proof}

\section{Morita Equivalence}

In this section we will give one application of the characterization of zigzag inverse semigroups: every inverse semigroup $S$ is Morita equivalent to a zigzag inverse semigroup. We use the construction in \cite{AfaraLawson} of an inverse semigroup $IM(S,I,p)$ from a set $I$ and a McAlister function $p: I \times I \to S$.

Let $S$ be an inverse semigroup with $0$ and define $p : E \times E \to S$ by
\[
p(e,f) = \begin{cases}
		e & \text{ if } e = f \\
		0 & \text{ otherwise}.
	\end{cases}
\]

One can quickly verify that $p$ is a McAlister function. It follows from \cite[Lemma 2.3]{AfaraLawson} that $RM(S,E,p)$ consists of triples $(e,s,f)$ such that $ss^* \leq e$ and $s^*s \leq f$. One can then use \cite[Lemma 2.6]{AfaraLawson} to see that $IM(S,E,p)$ can be identified as the inverse semigroup:
\[
IM(S,E,p) = \{ (e,s,f) : ss^* \leq e, s^*s \leq f, \text{ and } s \neq 0 \} \cup \{0\}
\]
with inversion given by  $(e,s,f) = (f,s^*,e)$ and multiplication 
\[
(e,s,f)(e',t,f') =  \begin{cases}
		(e, st, f') & \text{ if } f = e' \text{ and } st \neq 0 \\
		0 & \text{ otherwise}.
	\end{cases}
\]

\begin{thm} Let $S$ be an inverse semigroup and let $IM(S,E,p)$ be defined as above. Then $IM(S,E,p)$ is a zigzag inverse semigroup.
\end{thm}
\begin{proof}
We verify the conditions (Z1), (Z2), and (Z3).  The non-zero idempotents of $IM(S,E,p)$ are of the form $(e,x,e)$ where $e$ is idempotent in $S$ and $0 \neq x \leq e$. Moreover, we have $(e,x,e) \leq (f,y,f)$ if and only if $e = f$ and $x \leq y$. Thus the non-zero idempotent $(e,x,e)$ lies under the unique maximal idempotent $(e,e,e)$. Thus $IM(S,E,p)$ satisfies (Z1).

The paths in $IM(S,E,p)$ are the elements of the form $(e,s, s^*s)$ where $0 \neq ss^* \leq e$. Given $(e,s,f)$ in $IM(S,E,P)$ we have \[
    (e,s,f) = (e,s,s^*s)(s^*s, s^*s, f) = (e,s,s^*s)(f, s^*s, s^*s)^*.
\]
Therefore $IM(S,E,p)$ satisfies (Z2).

Next let $(e,s,f), (e',t,f') \in IM(S,E,p)$, and suppose that $P_{(e,s,f)} = P_{(e',t,f')}$ and $(e,s,f) z = (e',t,f') z$ for all $z \in P_{(e,s,f)}$. Notice that $P_{(e,s,f)} = \{ (f,x,x^*x) : 0 \neq xx^* \leq s^*s\}$. Thus $f = f'$. Letting $z = (f, s, s^*s)$ we see that $(e,s,s^*s) = (e', ts^*s, s^*s)$. Thus $e = e'$ and $s \leq t$. Similarly we get that $t \leq s$. Therefore $(e,s,f) = (e',t,f')$. We have that $IM(S,E,p)$ satisfies (Z3).
\end{proof}

One can quickly see that the path category of $IM(S,E,p)$ in the above proof is singly aligned. Therefore we have the following result.

\begin{cor} Every inverse semigroup is Morita equivalent to the zigzag inverse semigroup of some singly aligned category.
\end{cor}

In fact, the corollary can also be derived from Lawson's construction of a left cancellative category of an ordered groupoid with maximal identities, and the fact that the ordered groupoid of the category is an enlargement of the original groupoid \cite[Corollary 2.3.5]{LawsonCategory}.

In \cite{SteinbergMorita}, Steinberg shows that Morita equivalent inverse semigroups have Morita equivalent universal groupoids. He comments after Theorem 4.7 that the same mapping shows that they
also have Morita equivalent tight groupoids. Therefore the tight groupoid of any countable inverse semigroup is Morita equivalent to the tight groupoid of $\ZM(\C)$ for some singly aligned category $\C$. 
It follows by \cite[Corollary 7.10]{BKQS} that the class of Cuntz-Krieger $C^*$-algebras of singly aligned categories include the tight $C^*$-algebras of all countable inverse semigroups up to Morita equivalence.

\bibliographystyle{amsplain}
\bibliography{SemigroupBib.bib}

\providecommand{\bysame}{\leavevmode\hbox to3em{\hrulefill}\thinspace}
\providecommand{\MR}{\relax\ifhmode\unskip\space\fi MR }
\providecommand{\MRhref}[2]{%
  \href{http://www.ams.org/mathscinet-getitem?mr=#1}{#2}
}
\providecommand{\href}[2]{#2}
\begin{thebibliography}{10}

\bibitem{AfaraLawson}
B.~Afara and Mark~V. Lawson, \emph{Morita equivalence of inverse semigroups},
  Periodica Mathematica Hungarica \textbf{66} (2013), no.~1, 119--130.

\bibitem{AshHall}
C.~J. Ash and T.~E. Hall, \emph{Inverse semigroups on graphs}, Semigroup Forum
  \textbf{11} (1975/76), no.~2, 140--145.

\bibitem{BKQS}
Erik B\'{e}dos, S.~Kaliszewski, John Quigg, and Jack Spielberg, \emph{On
  finitely aligned left cancellative small categories, {Z}appa-{S}z\'{e}p
  products and {E}xel-{P}ardo algebras}, arXiv:1712.09432, December 2017.

\bibitem{CherubiniPetrich}
Alessandra Cherubini and Mario Petrich, \emph{The inverse hull of right
  cancellative semigroups}, J. Algebra \textbf{111} (1987), no.~1, 74--113.

\bibitem{DonsigMilan}
Allan Donsig and David Milan, \emph{Joins and covers in inverse semigroups and
  tight ${C}^*$-algebras}, Bull. Aust. Math. Soc. \textbf{90} (2014), no.~1,
  121--133.

\bibitem{ExelSteinberg}
R.~Exel and B.~Steinberg, \emph{Representations of the inverse hull of a
  $0$-left cancellative semigroup}, arXiv:1802.06281, February 2018.

\bibitem{LawsonCategory}
Mark~V. Lawson, \emph{Ordered groupoids and left cancellative categories},
  Semigroup Forum \textbf{68} (2004), no.~no. 3, 458--476.

\bibitem{LawsonGraph}
Mark~V. Lawson and David~G. Jones, \emph{Graph inverse semigroups: their
  characterization and completion}, J. Algebra (2014), 444--473.

\bibitem{Leech}
J.~Leech, \emph{Constructing inverse monoids from small categories}, Semigroup
  Forum \textbf{36} (1987), no.~1, 89--116.

\bibitem{spielberg}
Jack Spielberg, \emph{Groupoids and ${C}^*$-algebras for categories of paths},
  Trans. Amer. Math. Soc. \textbf{366} (2014), 5771--5819.

\bibitem{spielberg2}
\bysame, \emph{Groupoids and ${C}^*$-algebras for left cancellative small
  categories}, arXiv:1712.07720 [math.OA], December 2017.

\bibitem{SteinbergMorita}
B.~Steinberg, \emph{Strong {M}orita equivalence of inverse semigroups}, Houston
  J. Math. \textbf{37} (2011), no.~3, 895--927.

\end{thebibliography}
\end{document}